\documentclass{article}
\usepackage[utf8]{inputenc}
\usepackage[a4paper, total={6in, 10in}]{geometry}
\usepackage[T1]{fontenc}
\usepackage{amsmath,amsthm,amssymb}
\usepackage{caption,subcaption,graphicx}
\usepackage{enumerate}
\usepackage{authblk}
\usepackage[sorting=none]{biblatex}
\newtheorem{theorem}{Theorem}[section]

\newtheorem{lemma}[theorem]{Lemma}
\newtheorem{conjecture}{Conjecture}[]
\newtheorem{case}{Case}
\newtheorem{claim}{Claim}

\title{\textbf{Every Elementary Graph is Chromatic Choosable}}
\author{\textbf{Nandana K Vasudevan, K Somasundaram, J Geetha}}
\affil{Department of Mathematics\\ Amrita School of Physical Sciences - Coimbatore\\ Amrita Vishwa Vidyapeetham, India}
\date{\textit{Email : kv\_nandana@cb.students.amrita.edu, s\_sundaram@cb.amrita.edu \\ j\_geetha@cb.amrita.edu}}

\begin{document}

\maketitle

\begin{abstract}
Elementary graphs are graphs whose edges can be colored using two colors in such a way that the edges in any induced $P_3$ get distinct colors. They constitute a subclass of the class of claw-free perfect graphs. In this paper, we show that for any elementary graph, its list chromatic number and chromatic number are equal. 
\end{abstract}


\section{Introduction}
\hfill \par List coloring of graphs is a more general version of vertex coloring. It was introduced independently by Vizing and by Erdos, Rubin and Taylor. The \textit{list assignment} of a graph $G$ is a function $L:V(G) \rightarrow \mathcal{P}(\mathbb{N})$ that assigns to each vertex $v \in V(G)$, a list of colors $L(v)$. The list assignment $L$ is called a \textit{$k$-list assignment} if $|L(v)|=k, \forall v \in V(G)$. If $L$ is a list assignment of $G$, \textit{$L$-coloring} is a coloring in which each vertex receives a color from its list. It is proper if no two adjacent vertices receive the same color. We say $G$ is \textit{k-list colorable} or \textit{$k$-choosable} if $G$ has a proper $L$-coloring for all $k$-list assignment. The smallest $k$ for which $G$ is $k$-choosable is called \textit{list chromatic number}, \textit{choice number} or \textit{choosability} of $G$, denoted by $\chi_l(G)$, $\chi_L(G)$ or ch(G). We use the notation $\chi_l(G)$ throughout in this paper.  For any graph $G$ with chromatic number $\chi(G)$, $\chi_l(G) \geq \chi(G)$. This is called the \textit{choice chromatic inequality} \cite{gravier1998graphs}. The difference between $\chi_l(G)$ and $\chi(G)$ can be arbitrarily large.  For example, $\chi(K_{n,n^{n}})=2$ whereas $\chi_l(K_{n,n^{n}})=n+1$.   If $\chi_l(G)=\chi(G)$, then $G$ is called a \textit{chromatic choosable graph} \cite{ohba2002chromatic}. There is a natural question arises on when a graph $G$ is chromatic choosable. Many conjectures were proposed on this for different classes of graphs.  

\hfill \par A graph is said to be \textit{H-free} if it does not have the graph $H$ as an induced subgraph. If $H$ is the claw $K_{1,3}$, the graph is called \textit{claw-free}.  Gravier and Maffray \cite{gravier1998graphs} posed a question on the chromatic choosability of claw-free graphs, which they later stated as the List Coloring Conjecture (LCC) for claw-free graphs in \cite{gravier1997choice}.  

\begin{conjecture}[\textbf{LCC for claw-free graphs} \cite{gravier1997choice}]
Every claw-free $G$ satisfies $\chi_l(G)=\chi(G).$
\end{conjecture}

Esperet et al. \cite{esperet2014list} proved that for every claw-free graph $G$, if $\omega(G) \leq 3$ then $G$ is 4-choosable and if $\omega(G) \leq 4$ then $G$ is 7-choosable. Cranston and Rabern \cite{cranston2017list} proved several results on list coloring of claw-free graphs. In particular, they proved that every claw-free graph $G$ with $\Delta(G) \geq 69$ and $\omega(G)<\Delta(G)$
satisfies $\chi_l(G) \leq \Delta(G)-1$. LCC for claw-free graphs is still open. In 2016, Gravier, Maffray and  Pastor \cite{gravier2016choosability} narrow down the LCC for claw-free perfect graphs. 

\hfill \par A graph $G$ is said to be \textit{perfect} if $\chi(G)= \omega(G)$, where $\omega(G)$ is the clique number of $G$. The class of claw-free perfect graphs is an important subgraph class of claw-free graphs. Chvatal and Sbihi \cite{chvatal1988recognizing} have given a decomposition for the class of claw-free perfect graphs. Before we go to their decomposition, we state some definitions. A \textit{clique cutset} is a vertex cut that induces a clique. The construction of \textit{peculiar graph} is as follows. Let $K_n$ be a complete graph whose set of vertices is split into pairwise disjoint nonempty sets $A_1,B_1,A_2,B_2,A_3,B_3$. For each $i= 1, 2, 3,$ remove at least one edge with one endpoint in $A_i$ and the other endpoint in $B_{i+1}$ (subscript is modulo 3). Finally, add pairwise disjoint nonempty cliques $Q_1,Q_2,Q_3$, and, for each $i= 1, 2, 3$, make each vertex in $Q_i$ adjacent to all the vertices in $K_n-(A_i \cup B_i)$. An example of a Peculiar graph with all $A_i,B_i,Q_i$ as $K_1$ is shown in Fig. \ref{Peculiar graph}. A graph $G$ is said to be an \textit{elementary graph} if its edges can be colored with two colors such that every induced path of length 3 has its edges colored distinctly. Suppose the colors given to the edges are pink and green. Let the pink graph of a given elementary graph $G$ be the spanning subgraph of $G$ with edge set as those edges colored in pink and the green graph of $G$ be the spanning subgraph of $G$ with edge set as those edges colored in green. In Fig. \ref{Elementary graph with pink and green}, we have shown an elementary (Fig. \ref{Elementary graph}) graph and its corresponding pink graph (Fig. \ref{Elementary graph 1}) and green graph (Fig. \ref{Elementary graph 2}). The decomposition of the claw-free perfect graphs is as follows.

\begin{figure}[h]
    \centering
    \includegraphics[height=4cm]
    {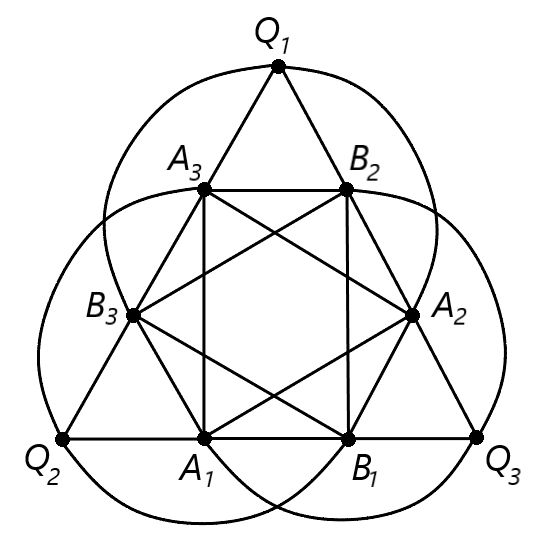}
    \caption{Peculiar graph with all $A_i,B_i,Q_i$ as $K_1$}
    \label{Peculiar graph}
\end{figure}

\begin{figure}[h]
    \begin{center}
    \begin{subfigure}{0.3 \textwidth}
        \centering
        \includegraphics[height=4cm]{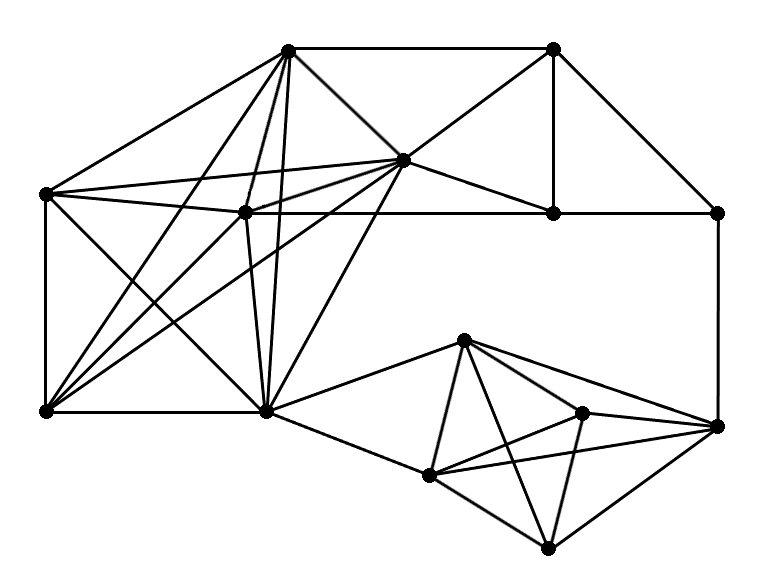}
        \caption{An elementary graph}
        \label{Elementary graph}
    \end{subfigure}
    \newline
    \begin{subfigure}{0.3 \textwidth}
        \centering
        \includegraphics[height=4cm]{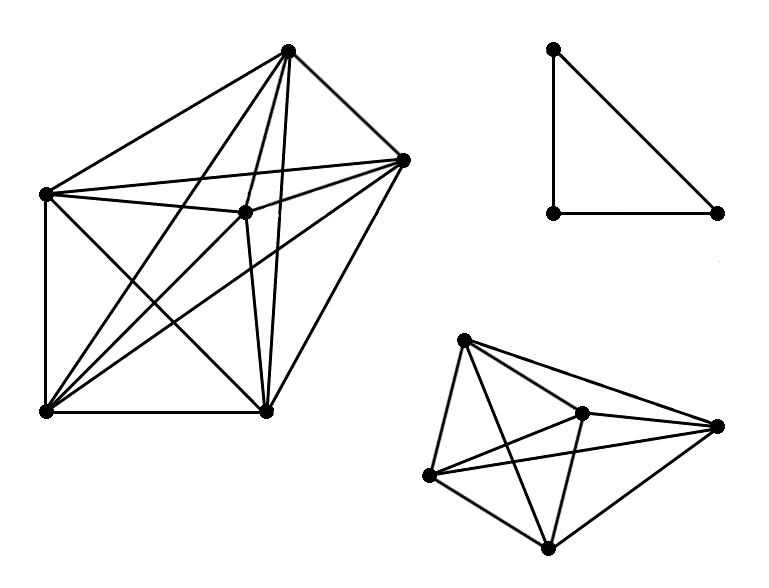}
        \caption{The pink graph}
        \label{Elementary graph 1}
    \end{subfigure}
    \hspace{3cm}
    \begin{subfigure}{0.3 \textwidth}
        \centering
        \includegraphics[height=4cm]{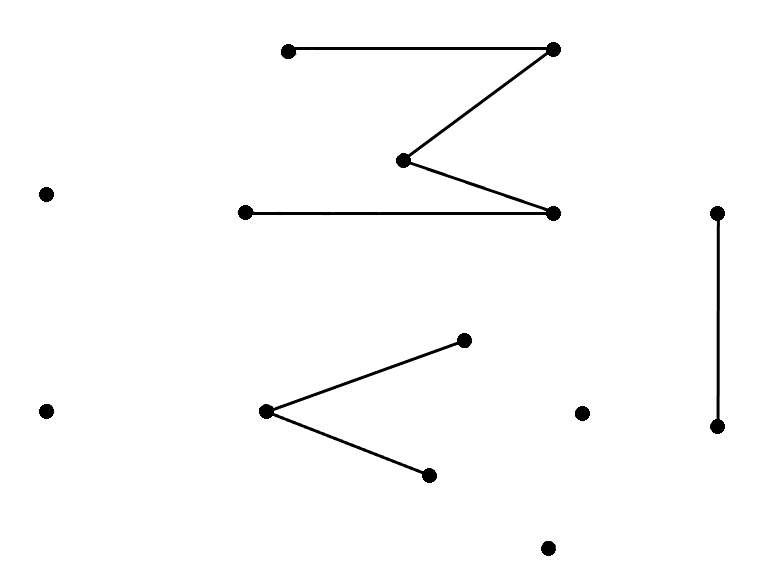}
        \caption{The green graph}
        \label{Elementary graph 2}
    \end{subfigure}  
    \caption{An elementary graph, its pink graph and its green graph}
    \label{Elementary graph with pink and green}
    \end{center}
\end{figure}

\begin{theorem}[\cite{chvatal1988recognizing}]
Every claw-free perfect graph with no clique cutset is either elementary or peculiar.
\end{theorem}

In 1997, Gravier and Maffray \cite{gravier1997choice} proved that any elementary graph $G$ with clique number at most 3 is chromatic choosable. The chromatic choosability of peculiar graphs had been given by the same authors in \cite{gravier1998graphs} during 1998. Later in 2004, they \cite{gravier2004choice} established that the chromatic choosability for any claw-free perfect graph with clique number less than or equal to 3. These works led to the question on the chromatic choosability of any claw-free perfect graph. In 2016,  Gravier et al. \cite{gravier2016choosability} stated it as a conjecture.

\begin{conjecture}[\textbf{LCC for claw-free perfect graphs} \cite{gravier2016choosability}]
Every claw-free perfect graph $G$ satisfies $\chi_l(G)=\chi(G).$
\end{conjecture}

The same authors proved their conjecture for claw-free graphs with clique number 4.
\begin{theorem}[\cite{gravier2016choosability}]
Let $G$ be a claw-free perfect graph with $\omega(G) \leq 4$. Then, $\chi_l(G)=\chi(G).$
\end{theorem}

\par Our focus is on the LCC for claw-free perfect graphs, as part of which, we have looked at list coloring of elementary graphs. In the next section, we provide a proof for the chromatic choosability of elementary graphs.

\section{Elementary Graphs}

\hfill \par There are different ways of defining an elementary graph. The structure of elementary graphs can be given using forbidden subgraphs. Before going to the definition of elementary graphs using forbidden subgraphs, let us see two definitions. An \textit{odd hole} in a graph $G$ is an induced subgraph of $G$ that is an odd cycle. An \textit{odd antihole} in a graph $G$ is an induced subgraph of $G$ that is the complement of an odd cycle.

\begin{theorem}[\cite{maffray1999description}]
A graph $G$ is an elementary graph if and only if it does not contain any of the eight graphs : claw, odd hole, odd antihole, pyramid, lighthouse, garden, colossus and mausoleum.
\end{theorem}

Another way of defining an elementary graph is using the Gallai graph. The \textit{Gallai graph} of a given graph $G$, \textit{Gal(G)} is the graph with vertex set as edge set of $G$ and edges are pairs of edges in $G$ which induces a $P_3$.

\begin{theorem}[\cite{maffray1999description}]
A graph $G$ is elementary if and only if \textit{Gal(G)} is bipartite.
\end{theorem}

\par A \textit{cobipartite graph} is the complement of a bipartite graph. An edge in a graph $G$ is called \textit{flat edge} if it does not lie in a triangle. Let $e=xy$ be a flat edge in the graph $G$ and $H=(X,Y, E_{XY})$ be a cobipartite graph which is disjoint from $G$. We assume $E_{XY}$ is non-empty and that $X \cup Y$ does not induce a clique. We construct a new graph by deleting the vertices $x,y$ and replacing them with $H$ such that each vertex in $X$ is adjacent to the vertices in $N(x)-\{y\}$ and each vertex in $Y$ is adjacent to the vertices in $N(y)-\{x\}$, where $N(x)$ and $N(y)$ denotes the neighbors of $x$ and $y$ respectively. Then, the resulting graph is called the \textit{augmentation} of $G$ along $xy$ and $H$ is called \textit{augment} of $xy$. The augmentation of a graph $G$ is the graph obtained by augmenting $h$ pairwise non-incident flat edges $x_1y_1, x_2y_2,...,x_hy_h$ by $h$ pairwise disjoint cobipartite graphs $H_1, H_2,...,H_h$. 

\begin{theorem} [\cite{maffray1999description}]\label{theorem1}
Let $G$ be a graph. Then, $G$ is elementary if and only if it is an augmentation of a line graph $L(B)$ of a bipartite multigraph $B$.
\end{theorem}

In Fig. \ref{Augmentation}, we have shown how the elementary graph in Fig. \ref{Elementary graph} has been expressed as an augmentation of a line graph of a bipartite multigraph.

\begin{figure}[h]
\centering
    \begin{subfigure}{0.3 \textwidth}
        \centering
        \includegraphics[height=2.8cm]{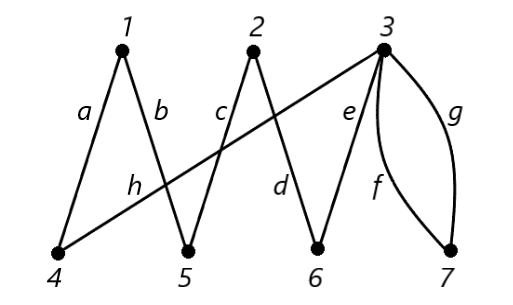}
        \caption{Bipartite multigraph $B$}
        \label{Bipartite multigraph}   
    \end{subfigure}
    \hspace{3cm}
    \begin{subfigure}{0.3 \textwidth}
        \centering
        \includegraphics[height=4.5cm]{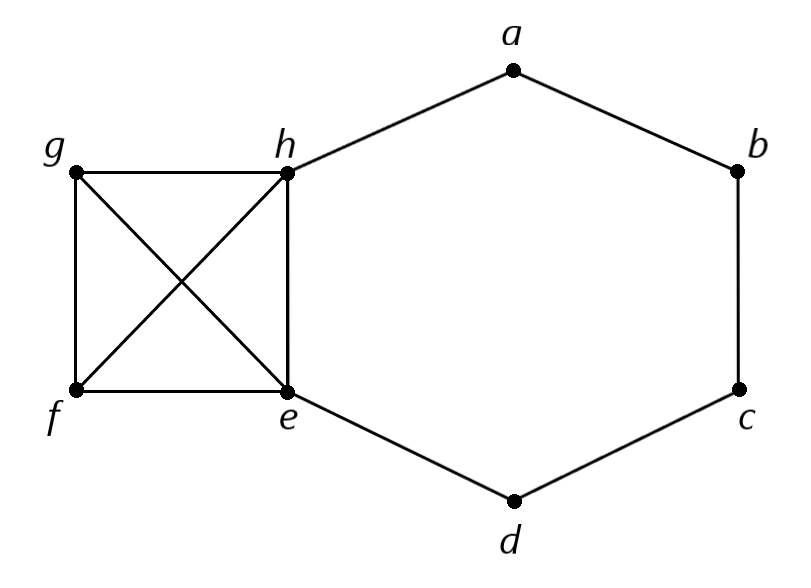}
        \caption{$L(B)$ the Line graph of $B$}
        \label{Line graph}  
    \end{subfigure}
    \newline
    \begin{subfigure}{0.3 \textwidth}
        \centering
        \includegraphics[height=6.5cm]{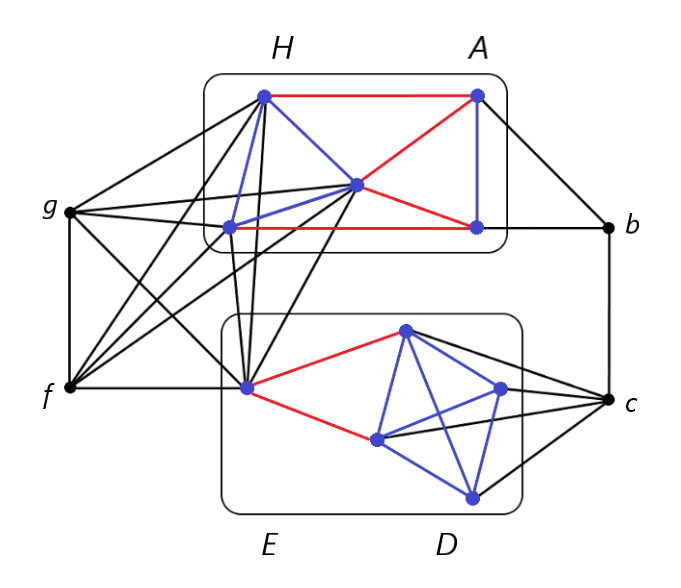}
        \caption{Augmentation of $L(B)$}
        \label{Augmented Elementary graph}   
    \end{subfigure}
    \caption{Elementary graph as an augmentation of a line graph of a bipartite multigraph}
    \label{Augmentation}
\end{figure}

\begin{lemma}[\cite{gravier2016choosability}]\label{lemma1}
Let $H=(X,Y,E_{XY})$ be a cobipartite graph such that $|X| \geq |Y|$ and there are $|Y|$ non-edges in $H$ and it forms a matching in $\overline{H}$. Let $L$ be a list assignment on $H$ such that $|L(x)| \geq |X|, \forall x \in X$ and $|L(y)| \geq |Y|, \forall y \in Y$. Then $H$ is $L$-colorable.
\end{lemma}

\begin{lemma}\label{lemma2}
Let $G$ be a graph and $L$ be a list assignment of $G$. If there exist a clique $Q$ such that $|L(Q)|<|Q|$, where $|L(Q)|=|\cup_{v \in Q}L(v)|$, then $G$ is not $L$-colorable. 
\end{lemma}

\renewcommand\qedsymbol{$\blacksquare$}

\begin{proof}    
Let us assume that $G$ is $L$-colorable. For any clique $Q$ in $G$, the vertices in $Q$ get distinct colors. Then,
\begin{center}
    $|L(Q)|=|\cup_{v \in Q}L(v)| \geq |\cup_{v \in Q}c(v)|=|Q|$
\end{center}
So, $G$ is $L$-colorable implies $|L(Q) \geq |Q|$, which is a contradiction.
\end{proof}

\begin{lemma}\label{lemma3}
Let $G(X, Y)$ be a bipartite graph with matching number $\mu(G)$. Then for the cobipartite graph $\overline{G}$, we have, 
\begin{center}
$\omega(\overline{G})=|X|+|Y|-\mu(G)$
\end{center}
\end{lemma}

\renewcommand\qedsymbol{$\blacksquare$}

\begin{proof}
By Gallai and Konig \cite{west2001introduction}, we have,
\begin{center}
$\alpha(G)+\mu(G)=|X|+|Y|$   
\end{center}
where $\alpha(G)$ is the independent number of $G$. As $\overline{G}$ is complement of $G$, we have $\omega(\overline{G})=\alpha(G)$. Substituting this in the above result, we get the desired result.
\end{proof}

Now, we give the proof of our main result.
\begin{theorem}
Let $G$ be an elementary graph. Then $\chi_l(G)=\chi(G)$.
\end{theorem}
\begin{proof}
As $G$ is an elementary graph, by Theorem \ref{theorem1}, it is an augmentation of the line graph of some bipartite multigraph $B$. Suppose $e_1,e_2,...,e_h$ are the flat edges in $L(B)$ that are augmented to obtain $G$. We prove the theorem using mathematical induction on $h$. 
\par For $h=0$, the graph has no augmentation and it is simply the line graph of some bipartite multigraph. In \cite{galvin1995list}, Galvin proved that the list chromatic index and chromatic index of a bipartite multigraph are equal. In other words, the list chromatic number and chromatic number of the line graph of a bipartite multigraph are equal. So the elementary graph with no augmentations is $\omega$-choosable.
\par Let us assume that for any elementary graph with at most $h-1$ augmentations, the theorem holds. Let $G$ be an elementary graph with $\omega(G)=\omega$ and has $h$ augmentations along edges $e_1,e_2,...,e_h$. Let $v_xv_{xy}$ and $v_yv_{xy}$ be the edges in $B$ whose corresponding edge is $e_h=xy$ in $L(B)$ and $H=(X,Y,E_{XY})$ be the $h^{th}$ augment of $G$. Let $L$ be an $\omega$-list assignment on $G$. By the hypothesis, $G-H$ is $L$-colorable. Let $c$ be an $L$-coloring of $G-(X \cup Y)$. We now list color $H$. 
\par In \cite{gravier2016choosability}, the authors have given proof for the list coloring of $H$ when $|X \cup Y| \leq \omega$. So, we check for $|X \cup Y|>\omega$. Let $N_X=N(X)-Y$ and $N_Y=N(Y)-X$. It can be seen that $N_X \cup X$ and $N_Y \cup Y$ induce cliques (Fig. \ref{G}). So, $|N_X \cup X| \leq \omega$ and $|N_Y \cup Y| \leq \omega$.\\
Define,
    $L'(u)=
    \begin{cases}
    L(u)-c(N_X), if \ u \in X \\ L(u)-c(N_Y), if \ u \in Y
    \end{cases}$.

So to show that $G$ is $\omega$-choosable, it is enough to find an $L'$-coloring for $G[X \cup Y]$.

\begin{figure}[h]
    \centering
    \includegraphics[height=5cm]{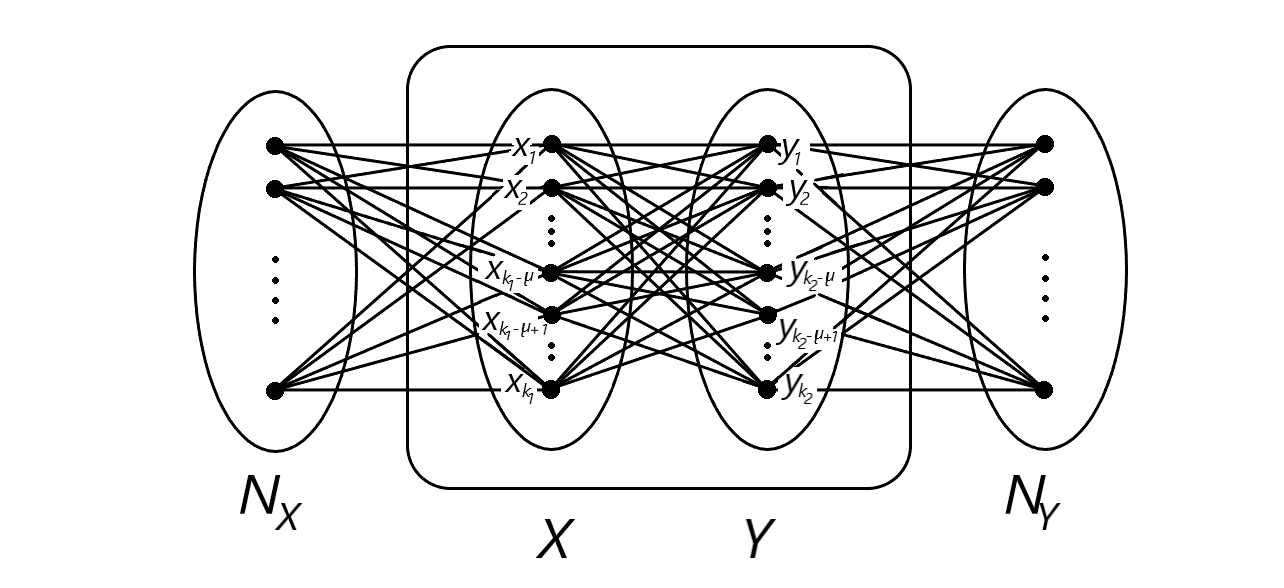}
    \caption{$X \cup Y$ in $G$}
    \label{G}
\end{figure}

\begin{case}
$|X| =\omega, |Y| \leq \omega$
{\fontfamily{} \selectfont
\par As $|X|=\omega, |Y| \leq \omega$, we have $|N_X|=0$ and $|N_Y| \geq 0$. Then, $|L'(x)|=\omega, \forall x \in X$ and $|L'(y)| \geq \omega, \forall y \in Y$.
Hence by Lemma \ref{lemma1}, $G[X \cup Y]$ can be $L'$-colored. 
}
\end{case}

\begin{case}
$|X|=k_1, |Y|=k_2 \ (k_2 \leq k_1 < \omega)$
{\fontfamily{} \selectfont
\par As $|X|=k_1,|Y|=k_2$, we have, 
\begin{center}
    $|N_X| \leq \omega-k_1$, $|N_Y| \leq \omega-k_2$ \\
    and $|L'(x)|\geq k_1, \forall x \in X$, $|L'(y)|\geq k_2, \forall y \in Y$
\end{center}
Assume that $G[X \cup Y]$ is not $L'$-colorable. Then, there exists a clique $Q'$ in $G[X \cup Y]$ that satisfies the inequality in Lemma \ref{lemma2}. Let $Q$ be a clique of size $\omega(G[X \cup Y])$ that contains $Q'$. If $Q$ is in $Y$, we have $|Q| \leq k_2$. But as $|L'(y)|\geq k_2, \forall y \in Y$, it would violate $|L'(Q)|<|Q|$ . So $Q$ is not in $Y$ and at least one vertex in $X$ is in $Q$. Then, we have,
\begin{center}
    $k_1 \leq |L'(Q)| < |Q| \leq \omega$
\end{center}
} 
\fontfamily{} \selectfont 
Suppose $|Q|=\omega'$. By Lemma \ref{lemma3}, we have $\mu=|X|+|Y|-\omega'$ where $\mu$ is the matching number of $\overline{G[X \cup Y]}$.

\begin{claim}
   $\mu<k_2$ 
\end{claim}
If $\mu \geq k_2$,
\begin{equation}
\begin{split}
|X| + |Y| - \omega'& \geq k_2\\
k_1 + k_2 - \omega'& \geq k_2 \\
k_1 - \omega'& \geq 0 \\
k_1 & \geq \omega'\\
\end{split}
\end{equation}
This is a contradiction to our assumption. \\

As $\mu<k_2$, there is at least one pair of vertices in $\overline{G[X \cup Y]}$, say $x_1,y_1$, which are not adjacent. That is, the edge $x_1y_1 \in Q$ for all $Q$ of size $\omega'$. If we take $|L'(Q)|=l$, then $|L'(x_1) \cup L'(y_1)| \leq l$. In other words, if $|(L(x_1)-c(N_X)) \cup (L(y_1)-c(N_Y))| \leq l$, then $G$ is not $L$-colorable.\\

Let $L(x_1)=\{c_1,c_2,...,c_{l_\cap}, c_{l_\cap+1},...,c_{l_1},c_{l_1+1},...,c_{\omega}\}$, $L(y_1)=\{c_1,c_2,...,c_{l_\cap}, c_{l_\cap+1}',...,c_{l_2}',c_{l_2+1}',...,c_{\omega}'\}$, $L(x_1) \cap c(N_X)=\{c_{l_1+1},...,c_{\omega}\}$ and $L(y_1) \cap c(N_Y)=\{c_{l_2+1}',...,c_{\omega}'\}$,where $l_\cap=|L(x_1) \cap L(y_1)|$. If $l_1+l_2-l_\cap \leq l$ then
 $G$ is not $L$-colorable.\\

\par Let us construct a new graph $G^*$ by replacing $X=\{x_1,x_2,...,x_{k_1}\}$ and $Y=\{y_1,y_2,...,y_{k_2}\}$ with $X^*=\{x_1^*,x_2^*,...,x_{k_1-\mu}^*\}$ and $Y^*=\{y_1^*,y_2^*,...,y_{k_2}^*\}$ respectively (see Fig.\ref{G*}), and $X^* \cup Y^*$ is a clique. It can be seen that $G^*$ is isomorphic to the graph obtained from $G$ by deleting the $\mu$ vertices $\{x_{k_1-\mu+1},x_{k_1-\mu+2},...,x_{k_1}\}$ in $X$. Let $B^*$ be a bipartite multigraph obtained from $B$ by duplicating $(k_1-\mu)$ times the edge $v_xv_{xy}$  and $k_2$ times the edge $v_yv_{xy}$.  $G^*$ can be obtained from $L(B^*)$ by augmenting $e_1,e_2,...,e_{h-1}$ edges with $h-1$ augments. Hence $G^*$ is an elementary graph with $h-1$ augmentations.

\begin{figure}[h]
    \centering
    \includegraphics[height=5cm]{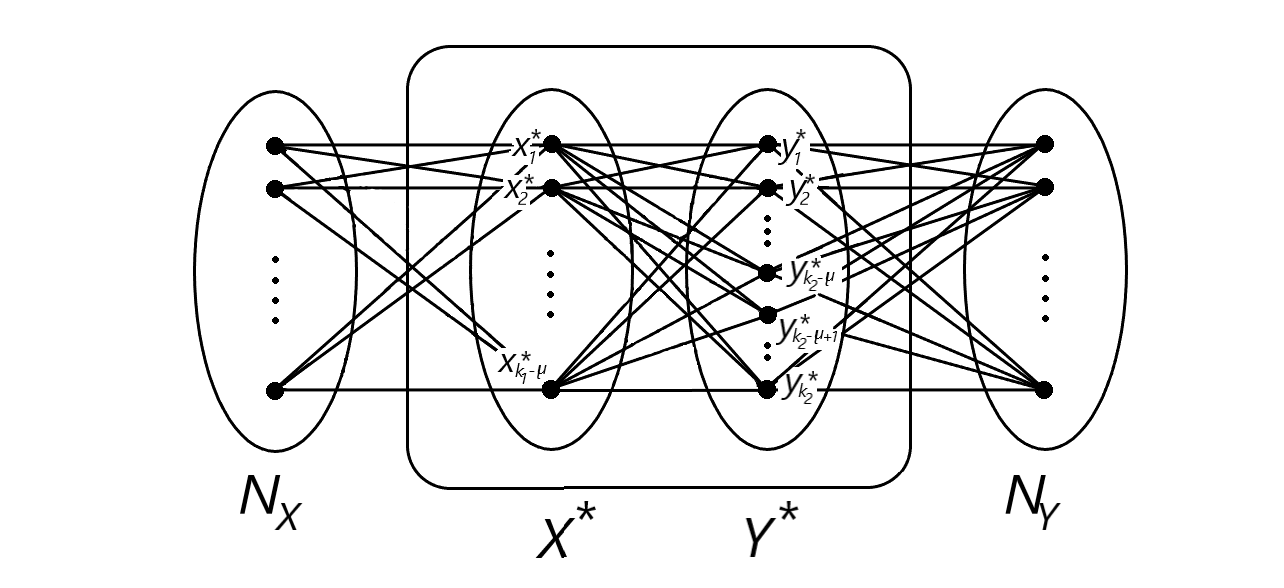}
    \caption{$X^* \cup Y^*$ in $G^*$}
    \label{G*}
\end{figure}

\par According to our induction hypothesis, $G^*$ is $\omega$-choosable. Let us define a list assignment $L^*$ as follows :
\begin{center}
$L^*(u)=\begin{cases}
L(u), u \in G^*-(X^* \cup Y^*)\\
L(x_1), u \in X^*\\
L(y_1), u \in Y^*
\end{cases}$.    
\end{center}
Let $c^*$ be an $L^*$-coloring of $G^*-(X^* \cup Y^*)$.\\

\begin{claim}
   $L^*(x_1) \cap c^*(N_X) \neq \{c_{l_1+1},...,c_{\omega}\}$ or $L^*(y_1) \cap c^*(N_Y) \neq \{c_{l_2+1}',..., c_{\omega}'\}$. 
\end{claim} 
\par Suppose $L^*(x_1) \cap c^*(N_X)=\{c_{l_1+1},..., c_{\omega}\}$ and $L^*(y_1) \cap c^*(N_Y)=\{c_{l_2+1}',..., c_{\omega}'\}$. Then, the vertices in $X^*$ will get their colors from $\{c_1,c_2,...,c_{l_\cap}, c_{l_\cap+1},...,c_{l_1}\}$ and the vertices in $Y^*$ will get their colors from $\{c_1,c_2,...,c_{l_\cap}, c_{l_\cap+1}',..., c_{l_2}'\}$. It is not possible to list color the vertices in $X^* \cup Y^*$ from this list as it induces a clique of size $\omega'$. This a contradiction to $G^*$ being $L^*$-colorable. So the claim holds and hence both conditions cannot occur simultaneously. \\

\par As we defined $L^*(u)=L(u)$ for $u \notin X \cup Y$, we can use $c^*$ instead of $c$ for these vertices. Hence the Claim 2 holds. This makes it possible to color $G[X \cup Y]$ and hence complete the $L$-coloring of $G$. Therefore, $G$ is $\omega$-choosable.
\end{case}
\end{proof}

In conclusion, every claw-free perfect graph is decomposed into peculiar, elementary, and graph with clique cutset. Gravier and Maffray proved every peculiar graph is chromatic choosable. In this paper, we proved that every elementary graph is chromatic choosable. Now, the case of graphs with clique cutset is open for the question of chromatic choosability. That is, whether a claw-free perfect graph with clique cutset is chromatic choosable.

\printbibliography[
heading=bibintoc,
title={Reference}
]

\end{document}